\documentclass[11pt,bezier]{article}  % include bezier curves
\usepackage{color}              % Need the color package
\usepackage{epsfig}
\usepackage{graphicx}
\usepackage{amsmath}
\usepackage{amsfonts}
\usepackage{amsthm}
\usepackage{amssymb}
\usepackage{epsfig}
\usepackage{hyperref}
\usepackage{bbm}
\definecolor{MyDarkBlue}{rgb}{0,0.08,0.50}
\definecolor{BrickRed}{rgb}{0.65,0.08,0}

\hypersetup{
colorlinks=true,       % false: boxed links; true: colored links
    linkcolor=MyDarkBlue,          % color of internal links
    citecolor=BrickRed,        % color of links to bibliography
    filecolor=red,      % color of file links
    urlcolor=cyan           % color of external links
}

\definecolor{darkgreen}{rgb}{0,.4,0}
\definecolor{darkagenta}{rgb}{.5,0,.5}
\definecolor{darkred}{rgb}{1,0,0}%was 0.85
\definecolor{darkblue}{rgb}{0,0,.4}

\newtheorem{Lemma}{Lemma}[section]
\newtheorem{Proposition}[Lemma]{Proposition}

\newtheorem{Theorem}[Lemma]{Theorem}

\newtheorem{Corollary}[Lemma]{Corollary}

\def\1{{\mathchoice {1\mskip-4mu\mathrm l}      % Blackboard bold 1
{1\mskip-4mu\mathrm l}
{1\mskip-4.5mu\mathrm l} {1\mskip-5mu\mathrm l}}}
\newcommand{\indic}[1]{\1_{\{#1\}}}

\newcommand{\prob}{\mathbb{P}}

\newcommand{\E}{\mathbb{E}}

\newcommand{\eps}{\varepsilon}

\newcommand{\Rbold}{{\mathbb{R}}}
\newcommand{\Zbold}{{\mathbb{Z}}}

\newcommand{\expec}{\mathbb{E}}

\newcommand{\eqn}[1]{\begin{equation} #1 \end{equation}}
\newcommand{\eqan}[1]{\begin{align} #1 \end{align}}
\newcommand{\lbeq}[1]{\label{#1}}
\newcommand{\refeq}[1]{(\ref{#1})}
\newcommand{\sss}{\scriptscriptstyle}

\newcommand {\convp}{\stackrel{\sss {\mathbb P}}{\longrightarrow}}

\newcommand {\vep}{\varepsilon}
\newcommand{\conn}{\longleftrightarrow}
\newcommand{\nn}{\nonumber}

\numberwithin{equation}{section}
\newcommand{\e}{{\rm e}}

\begin{document}
\author{Maria Deijfen \thanks{Department of Mathematics,
Stockholm University, 106 91 Stockholm, Sweden.
Email: {\tt mia@math.se}}
\and
Remco van der Hofstad
\thanks{Department of Mathematics and
Computer Science, Eindhoven University of Technology, P.O.\ Box 513,
5600 MB Eindhoven, The Netherlands. Email: {\tt
rhofstad@win.tue.nl}}
\and
Gerard Hooghiemstra
\thanks{DIAM, Delft University of Technology, Mekelweg 4, 2628CD Delft, The
Netherlands. Email: {\tt g.hooghiemstra@tudelft.nl} }
}

\title{Scale-free percolation}

\maketitle

\begin{abstract}

\noindent We formulate and study a model for inhomogeneous long-range percolation on $\Zbold^d$.
Each vertex $x\in\Zbold^d$ is assigned a non-negative weight $W_x$, where $(W_x)_{x\in\Zbold^d}$ are i.i.d.\ random variables. Conditionally on the weights, and given two parameters $\alpha,\lambda>0$,
the edges are independent and the probability that there is an edge between $x$ and $y$ is given by $p_{xy}=1-\exp\{-\lambda W_xW_y/|x-y|^\alpha\}$. The parameter $\lambda$ is the percolation parameter, while $\alpha$ describes the long-range nature of the model. We focus on the degree distribution in the resulting graph, on whether there exists an infinite component and on graph distance between remote pairs of vertices.

First, we show that the tail behavior of the degree distribution is related to the tail behavior of the weight distribution. When the tail of the distribution of $W_x$ is regularly varying with exponent $\tau-1$, then the tail of the degree distribution is regularly varying with exponent $\gamma=\alpha(\tau-1)/d$. The parameter $\gamma$ turns out to be crucial for the behavior of the model. Conditions on the weight distribution and $\gamma$ are formulated for the existence of a critical value $\lambda_c\in(0,\infty)$ such that the graph contains an infinite component when $\lambda>\lambda_c$ and no infinite component when $\lambda<\lambda_c$. Furthermore, a phase transition is established for the graph distances between vertices in the infinite component at the point $\gamma=2$, that is, at the point where the degrees switch from having finite to infinite second moment.

The model can be viewed as an interpolation between long-range percolation and models for inhomogeneous random graphs, and we show that the behavior shares the interesting features of both these models.

\vspace{0.5cm}


\vspace{0.5cm}

\end{abstract}

\section{Introduction}
\label{sec-intro}

The field of percolation has been very active the last few decades with important progress on questions concerning for instance the appearance and uniqueness of an infinite component and the decay of connectivity functions. In parallel, the area of random graphs has developed from dealing mainly with simple models with little structure to studying more complex models aimed at describing real-world networks. A particular class of graph models that has received substantial attention consists of \emph{inhomogeneous random graphs}, where the edge probabilities are defined in terms of weights that are associated to the vertices.
In the current paper, we combine the above two fields by introducing a model for spatial inhomogeneous random graphs on $\Zbold^d$ with long-range edges and vertex weights. We characterize the degree structure in the graph, determine when there is a non-trivial percolation threshold and prove a phase transition for the graph distance at the point where the variance of the degrees goes from being finite to infinite. Such a phase transition has already been established for several non-spatial models, and the fact that it appears also in the presence of spatial influence gives further support to the belief that it is a universal feature.

We define our model on the lattice $\Zbold^d$, where the integer $d\ge 1$ denotes the dimension.
Let each vertex $x\in\Zbold^d$ be equipped with a weight $W_x$, where $(W_x)_{x\in \Zbold^d}$ are
independent and identically distributed (i.i.d.). Conditionally on the weights
$(W_x)_{x\in \Zbold^d}$, the edges in the graph are independent and the
probability that there is an edge between $x$ and $y$ is defined by
    \eqn{
    \label{edgeprob}
    p_{xy}=1-\e^{-\lambda W_xW_y/|x-y|^{\alpha}},
    }
for $\alpha,\lambda\in(0,\infty)$. We say that the edge $(x,y)$ is \emph{occupied} with probability $p_{xy}$ and \emph{vacant} otherwise. The parameter $\alpha>0$ describes the long-range nature of our model, while we think of $\lambda>0$ as a percolation parameter. Naturally, the model
for fixed $\lambda>0$ and weights $(W_x)_{x\in \Zbold^d}$ is the same as the one for $\lambda=1$ and $(\sqrt{\lambda} W_x)_{x\in \Zbold^d}$, so there might appear to be some redundancy
in the parameters of the model. However, we view the the weights $(W_x)_{x\in \Zbold^d}$ as
creating a \emph{random environment} in which we study the percolative properties of the model.
Thus, we think of the random variables $(W_x)_{x\in \Zbold^d}$ as fixed once and for all
and we change the percolation configuration by varying $\lambda$. We can thus view our model
as percolation in a random environment given by the weights $(W_x)_{x\in \Zbold^d}$.

\paragraph{The choice of weight variables.}
The distribution of the weight variables $(W_x)_{x\in \Zbold^d}$ is clearly very important
for the properties of the model. When the weight variables have unbounded support,
vertices with very high vertex weight will be present. These vertices play a special
role, as they are much more likely to have a large number of edges emerging from them,
that is, vertices with high weight tend to have high degrees. In many real-world networks,
such vertices with high degrees are present. These form the \emph{hubs} of the network and often play a crucial role in the functionality of the network. Therefore, we are particularly
interested in settings where the weights are heavy tailed.

Our model has close links both to long-range percolation
(arising when $W_x\equiv 1$ for every $x\in \Zbold^d$) and to inhomogeneous random graphs
(arising when we consider the model on
a fixed number of vertices $\{1,\ldots, n\}$ and when $|x-y|^{\alpha}$ in \eqref{edgeprob}
is replaced by a simple factor $n$). We next discuss these models in more detail.

\paragraph{Long-range percolation.}
In long-range percolation, in the most common setup, two vertices $x,y\in\Zbold^d$ are connected by an edge with a probability that decays like $\lambda|x-y|^{-\alpha}$, for some parameters $\alpha,\lambda>0$, as $|x-y|\to\infty$, and the occupation statuses of different edges are independent random variables. In $d=1$, the percolation properties of the model depend on the value of $\alpha$: If $\alpha<1$, the graph is almost surely connected \cite{Sch}, if $\alpha\in(1,2)$, the graph contains an infinite component as soon as the nearest-neighbor edge probability is large enough \cite{NewSch} and, if $\alpha>2$, the graph contains only finite components. For $\alpha=2$, the behavior is the same as for $\alpha\in(1,2)$ when $\lambda>1$ while there are only finite components when $\lambda<1$; see \cite{AN}. Uniqueness of an infinite component in any $d\geq 1$ follows from the main result in \cite{GKN}. In $d\geq 2$, there is a non-trivial critical value, since already the nearest-neighbor connections are sufficient for the possibility of an infinite component. The attention there focuses on the effect of the long-range connections on the critical behavior and on the properties of the infinite component.

As for the graph distance in long-range percolation, Benjamini et al.\ \cite{BKPS} show that $d(0,x)$ is
bounded as $|x|\to\infty$ when $\alpha<d$ and Coppersmith et al.\ \cite{CGS} show for a version of the model where all nearest-neighbor connections are present that $d(0,x)$ grows like $\log|x|/\log\log|x|$ when $\alpha=d$. Furthermore, conditionally on that $0$ and $x$ are in the infinite component, Biskup \cite{Biskup} shows that $d(0,x)$ grows like $(\log|x|)^\Delta$ for an explicit $\Delta>1$ when $\alpha\in(d,2d)$ and Berger \cite{Berger} shows that it grows at least like $|x|$ when $\alpha>2d$. We mention also the model by Yukich \cite{Yukich07}, where each point $x\in\Zbold^d$ is assigned an i.i.d.\ weight $U_x^{-p}$, with $U_x$ uniformly distributed on $[0,1]$ and $p\in(1/d,\infty)$, and two points $x$ and $y$ are then connected if and only if $|x-y|\leq \min\{U_x^{-p},U_y^{-p}\}$. Since $U_x^{-p}\geq 1$, the graph is connected and Yukich shows that $d(0,x)$ grows at most like $\log\log|x|$ as $|x|\to\infty$.
The model is related to the Poisson Boolean model on $\Rbold^d$; see Section 6 for details.

\paragraph{Inhomogeneous random graphs.}
In inhomogeneous random graphs, the edges are conditionally independent, given some vertex weights. One example is the Poissonian random graph \cite{NorrReit}, where each vertex $i$ in a set of $n$ vertices is assigned a random weight $W_i$ and two vertices $i$ and $j$ are then connected by an edge if a Poisson variable with mean $W_iW_j/\sum_{k=1}^n W_k$ takes on a positive value. We mention also the expected degree model by Chung and Lu \cite{CL:1,CL:2} and the related generalized random graph \cite{BDM-L}, which are in the same universality class as the Poissonian random graph model. The asymptotic degree distributions in these graphs are determined by the distribution of the weights in that, if the weight distribution is regularly varying, then the degree distribution varies regularly with the \emph{same} exponent. As for the graph distance, the above models have all been proved to have a phase transition at the point where the degrees go from having finite to infinite second moment: The distances grow logarithmically when the degrees have finite second moment, and doubly logarithmically when the second moment is infinite. This has also been established for the well-known configuration model \cite{HHM,HHZ} and for preferential attachment models \cite{DomHofHoo10}. It is believed to be true for a large class of random graph models. Finally, we mention that many of the above models are special cases of the very general model treated in the seminal paper by Bollob\'as et al.\ \cite{BRJ}.

Our model \emph{interpolates} between long-range percolation and inhomogeneous random graphs in that it has both \emph{geometry} in a similar way as in long-range percolation, as well as \emph{vertex weights}, in a similar way as for certain inhomogeneous random graphs. The main message of this paper is that our model
inherits the interesting features of both models it interpolates between.

\paragraph{Organization and results.} This paper is organized as follows. In Section \ref{sec-vert}, we characterize the tail behavior of the degree distribution. Take the weight distribution to be regularly varying with exponent $1-\tau$, that is, $\prob(W>w)=w^{-(\tau-1)}L(w)$, where $w\mapsto L(w)$ is slowly varying at infinity. We show that the corresponding degree distribution is then regularly varying with exponent $-\gamma$, where $\gamma=\alpha(\tau-1)/d$, provided that $\alpha>d$ and $\gamma>1$. Note that, when $\gamma>2$, the degrees have finite variance, while when $\gamma\in (1,2]$, the degrees have finite mean, but infinite variance. Whether the degrees have infinite variance for $\gamma=2$ depends on the precise shape of the slowly varying function involved. When $\alpha\leq d$ or $\gamma\leq 1$, it is not hard to see that the model is degenerate in the sense that all vertices will have infinite degree almost surely; see Theorem \ref{thm:div_deg}.

In Sections \ref{sec-perc} and \ref{sec-pos}, the percolation theoretical properties of the model
are studied. To this end, we assume that $\expec[W]=1$ as soon as the mean weight is finite and view $\lambda>0$ as the percolation parameter. The critical value is denoted $\lambda_c$. In Section \ref{sec-perc}, conditions are formulated on the degree distribution that guarantee that $\lambda_c<\infty$
and, in Section \ref{sec-pos}, it is shown that $\lambda_c>0$ if and only if the degrees have finite variance.

Section \ref{sec-dist} investigates graph distances between vertices. Let $d(x,y)$ denote the graph distance between $x$ and $y$, that is, the minimal number of occupied edges that form a path between $x$ and $y$. When there is an infinite component in the graph and 0 and $x$ are both in this component, how does $d(0,x)$ grow with $|x|$? We show that $d(0,x)$ is at least of the order $\log|x|$ when $\gamma>2$, that is, when the degrees have finite variance, and exactly of the order $\log\log|x|$ when $\gamma<2$, that is, when the degrees have infinite variance. This establishes a phase transition at the point where $\gamma=2$. We improve the lower bound on the distances for $\gamma>2$ in the case where $\alpha>2d$ to $|x|^{\eps}$ for some $\eps>0$, which mimics the results for long-range percolation. Indeed, there the distances are polylogarithmic when $\alpha\in (d,2d)$ and polynomial when $\alpha>d$.

The present work gives rise to many interesting further
questions and, in Section \ref{sec-disc}, we give some suggestions.

\section{Vertex degrees}
\label{sec-vert}

Throughout the paper we assume that  the edge probabilities $(p_{xy})_{x,y\in {\mathbb Z}^d}$ are as in \refeq{edgeprob}, where the weights $(W_x)_{x\in {\mathbb Z}^d}$ are i.i.d. In this section we relate the tail behavior of the degree distribution in our model to that of the weight distribution. To this end, assume that the distribution $F$ of the weights $(W_x)_{x\in\Zbold^d}$ has a regularly varying tail with exponent $\tau-1$, that is, denoting by $W$ a random variable with the same distribution as $W_0$ and by $F$ its distribution function, we assume that
    \eqn{
    \label{weight-distr}
    1-F(w)=\prob(W>w)=w^{-(\tau-1)}L(w),
    }
where $w\mapsto L(w)$ is a function that varies slowly at infinity. Write $D_x$ for the degree of $x\in\mathbb{Z}^d$ and note that, by translation invariance, $D_x$ has the same distribution as $D_0$.

In this section, we prove two main results. Firstly, we show in Theorem \ref{thm:div_deg}
that, as soon as the weight of a vertex is positive, its degree is almost surely infinite
when $\alpha\leq d$ or when both $\alpha>d$ and $\gamma=\alpha(\tau-1)/d\leq 1$. Secondly, in Theorem \ref{th:deg}, we show that the degrees have a power-law distribution with exponent $\gamma=\alpha(\tau-1)/d$ when $\alpha>d$ and $\gamma>1$.

\begin{Theorem}[Infinite degrees for $\alpha\leq d$ or $\gamma\leq 1$]
\label{thm:div_deg}
Fix $d\geq 1$. Then, $\prob(D_0=\infty|W_0>0)=1$ when
$\alpha\leq d$, or when $\alpha>d$ and the weight distribution satisfies
    \eqn{
    \lbeq{power-law-lbd}
    1-F(w)\geq c w^{-(\tau-1)}, \qquad w\ge 0,
    }
for some $c>0$ and $\tau>1$ such that $\gamma=\alpha(\tau-1)/d\leq 1$.
\end{Theorem}
%%%%%%%%%%%%%%%%%%%%%%%%%%%%%%%%%%%%%%%%%
\proof
Denote the minimum (respectively, maximum) of two real numbers $x$ and $y$ by $x\wedge y$ (respectively, $x\vee y$). For $x,y\in \Zbold^d$, recall that $(x,y)$ is occupied if the edge between $x$ and $y$ is present in the graph. Using the bound $1-\e^{-x}\geq (x\wedge 1)/2$, we get
\begin{equation}\label{sum_occ}
\sum_{y\neq 0}\prob((0,y)\mbox{ occupied}|W_0=w)\geq \frac{1}{2}\sum_{y\neq 0}
\E\left[\frac{\lambda wW_y}{|y|^\alpha}\wedge 1\right]\geq \frac{\lambda w}{2}\sum_{y\neq 0}
\frac{\E\left[W_y\mathbf{1}_{\{W_y\leq |y|^\alpha w^{-1}\}}\right]}{|y|^\alpha},
\end{equation}
where $\mathbf{1}_A$ denotes the indicator of the event $A$.
As for (a), just note that clearly $\E\left[W_y\mathbf{1}_{\{W_y\leq |y|^\alpha w^{-1}\}}\right]\to \expec[W]$ as $|y|\to\infty$, implying that we can bound
$$
\sum_{y\neq 0}\prob((0,y)\mbox{ occupied}|W_0=w)\geq Cw\sum_{y\neq 0}\frac{1}{|y|^\alpha}
$$
for some constant $C>0$. If $\alpha\leq d$, then the sum in the bound diverges and, since the edges of the origin are independent conditionally on $W_0$, it then follows from the Borel-Cantelli lemma that $\prob(D_0=\infty|W_0=w)=1$ for every $w>0$. This implies that $\prob(D_0=\infty|W_0>0)=1$.

As for (b), if the weight distribution satisfies \refeq{power-law-lbd} and $\gamma=\alpha(\tau-1)/d\leq 1$, then $\tau\in (1,2]$. Thus, $\expec[W_y]=\infty$ and we obtain that
$$
\E[W_y\mathbf{1}_{\{W_y\leq s\}}]\geq C's^{2-\tau}.
$$
Combining this bound with (\ref{sum_occ}) yields
$$
\sum_{y\neq 0}\prob((0,y)\mbox{ occupied}|W_0=w)\geq C''w^{\tau-1}\sum_{y\neq 0}\frac{1}{|y|^{\alpha(\tau-1)}}.
$$
By the argument above, we have $\prob(D_0=\infty|W_0>0)=1$ as soon as $\gamma=\alpha(\tau-1)/d\le 1$.
\qed

\begin{Theorem}[Power-law degrees for power-law weights]
\label{th:deg} Fix $d\geq 1$. Assume that the weight distribution satisfies \refeq{weight-distr} with
$\alpha>d$ and $\gamma=\alpha(\tau-1)/d>1$. Then,
there exists $s\mapsto \ell(s)$ which is slowly varying at infinity such that
\eqn{
\label{degree-tail}
\prob(D_0>s)=s^{-\gamma}\ell(s).
}
\end{Theorem}

Under the assumptions of the theorem, the degrees have finite mean, that is, $\gamma=\alpha(\tau-1)/d>1$. Furthermore, it is easy to see that, for $\alpha>d$, finite variance for the weights (i.e., $\tau>3$) implies finite variance for the degrees (i.e., $\gamma>2$). Note however that the variance of the degrees may be finite even if the weights have infinite variance, since for a given value of $\tau\in(2,3)$ we have $\gamma>2$ if $\alpha$ is large enough.

In the remainder of this section, we prove Theorem \ref{th:deg}.
Write $v_d$ for the volume of the unit ball in $\Rbold^d$ and let $\Gamma(\cdot)$ denote the gamma function.
The proof of Theorem \ref{th:deg} relies on the following characterization of the conditional expected degree.

\begin{Proposition}[Asymptotic expected vertex degree]
\label{prop:exp_deg_bd}
Assume that the weight distribution satisfies (\ref{weight-distr}) with $\alpha>d$ and $\gamma>1$. Then
$$
|\expec[D_0|W_0=w]-\xi w^{d/\alpha}|\leq C,
$$
where  $\xi=\lambda^{d/\alpha}v_d\Gamma\left(1-\frac{d}{\alpha}\right)\E[W^{d/\alpha}]$ and $C=C(d)$ is a constant.
\end{Proposition}

We remark that if $\gamma>1$, then $\tau-1=\gamma d/\alpha>d/\alpha$, so that $\E[W^{d/\alpha}]<\infty$.

\proof
Observe that, given $W_0$, the degree $D_0$ is a sum of independent indicators and
$$
\expec[D_0|W_0=w]=\sum_{y\neq 0}
\Big(
1-\E\Big[\e^{-\lambda wW_y/|y|^{\alpha}}\Big]
\Big)
=
\sum_{y\neq 0}
\int_0^\infty
\Big(
1-\e^{-\lambda wu/|y|^{\alpha}}\Big)\, dF(u).
$$
We interchange the order of integration and summation and first compute
the sum over $y\neq 0$. To this end, write
\eqn
{
\sum_{y\neq 0}\Big(1-\e^{-\lambda wu |y|^{-\alpha}}\Big)=
\int_{|y|>1}\Big(1-\e^{-\lambda wu |y|^{-\alpha}}\Big)\,dy+ E_1(u),
}
where $E_1(u)$ is an error term that will be estimated below. A change of variables $y=(\lambda uw)^{1/\alpha}t$ yields that
\begin{eqnarray}
\int_{|y|>1}
(1-\e^{-\lambda wu |y|^{-\alpha}})\,dy&=&
(uw\lambda)^{d/\alpha} \int_{|t|>(uw\lambda)^{-1/\alpha}} (1-\e^{-|t|^{-\alpha}})\,dt\nonumber\\
&=&
    (uw\lambda)^{d/\alpha} \int_{|t|>0} (1-\e^{-|t|^{-\alpha}})\,dt-E_2(u),
    \label{int-ben}
\end{eqnarray}
where again $E_2(u)$ is an error term that will be dealt with below. Converting to polar coordinates followed by partial integration and finally a change of variables yields that
\begin{eqnarray*}
\int_{|t|>0} (1-\e^{-|t|^{-\alpha}})\,dt&=& v_d \int_{r=0}^\infty \Big(1-\e^{-r^{-\alpha}}\Big)\,d(r^d)
=v_d\int_{r=0}^\infty r^d \,d(\e^{-r^{-\alpha}})\\
&=&-v_d\int_{0}^\infty s^{-d/\alpha} \,d(\e^{-s})
=v_d\Gamma\Big(1-\frac{d}{\alpha}\Big),
\end{eqnarray*}
for $\alpha>d$. Hence, provided that $\expec[W^{d/\alpha}]<\infty$, which holds for $\gamma>1$, we obtain
\begin{eqnarray}
\expec[D_0|W_0=w]&=&
v_d\Gamma\Big(1-\frac{d}{\alpha}\Big)\int_0^\infty (uw\lambda)^{d/\alpha}\,dF(u)
+\int_0^\infty (E_1(u)-E_2(u))dF(u)\nonumber\\
&=&\xi w^{d/\alpha}+\int_0^\infty (E_1(u)-E_2(u))dF(u)\label{eq:exp_bd},
\end{eqnarray}
where $\xi=\lambda^{d/\alpha}v_d\Gamma\left(1-\frac{d}{\alpha}\right)\E[W^{d/\alpha}]$.

It remains to bound the error terms. As for $E_1(u)$, since $1-\e^{-c|y|^{-\alpha}}$ is monotonically decreasing as $|y|$ increases, we can estimate
$$
0\leq E_1(u)=\sum_{y\neq 0} \Big(1-\e^{-\lambda wu |y|^{-\alpha}}\Big) -\int_{|y|>1}
\Big(1-\e^{-\lambda wu |y|^{-\alpha}}\Big)\,dy \leq v_d.
$$
Moving on to $E_2(u)$, we have
$$
E_2(u)=(uw\lambda)^{d/\alpha} \int_{|t|\leq (uw\lambda)^{-1/\alpha}} (1-\e^{-|t|^{-\alpha}})\,dt,
$$
and a similar computation as the one following \eqref{int-ben} yields
$$
E_2(u)=v_d
\left\{
(1-\e^{-uw\lambda})+(uw\lambda)^{d/\alpha}\int_{uw\lambda}^\infty s^{-d/\alpha}\e^{-s}\,ds
\right\}.
$$
For $a\leq 1$ we have that
$$
\int_z^\infty s^{a-1}\e^{-s}\,ds \leq z^{a-1}\int_z^{\infty}\e^{-s}ds=z^{a-1}\e^{-z}.
$$
With $-d/\alpha=a-1$, it follows that
$$
(uw\lambda)^{d/\alpha}\int_{uw\lambda}^\infty s^{-d/\alpha}\e^{-s}\,ds\leq \e^{-uw\lambda},
$$
and hence $0\leq E_2(u)\leq v_d$. Combining these estimates with (\ref{eq:exp_bd}) yields
$$
|\expec[D_0|W_0=w]- \xi w^{d/\alpha}|\leq v_d.
$$
\qed

With Proposition \ref{prop:exp_deg_bd} at hand we proceed to prove Theorem \ref{th:deg}.

\begin{proof}[Proof of Theorem \ref{th:deg}]
We first give a heuristic argument. The tail of the degree distribution is obtained as
\eqn{
\label{survival}
\prob(D_0>s)
=\int \prob(D_0>s|W_0=w)\,dF(w).
}
It follows from Proposition \ref{prop:exp_deg_bd} that
\eqn{
\label{hypothesis}
\expec[D_0|W_0=w]=\xi w^{d/\alpha}+O(1),
}
as $w\to\infty$. Since $D_0|W_0=w$ is a sum of independent indicators, it is reasonable to expect that $\prob(D_0>s|W_0=w)$ is well approximated by the indicator function
$$
{\bf 1}_{\{\expec[D_0|W_0=w]>s\}}\approx {\bf 1}_{\{\xi w^{d/\alpha}>s\}}.
$$
This results in
$$
\prob(D_0>s)
\approx\int_{(s/\xi)^{\alpha/d}}^\infty \,dF(w)=
[1-F]((s/\xi)^{\alpha/d})=s^{-\alpha(\tau-1)/d}\ell(s),
$$
where $s\mapsto \ell(s)$ is slowly varying at infinity.

To formalize the above, we adapt the proof of \cite[Theorem 1.1]{Yukich07}. First, for fixed $s$, split the integral in \eqref{survival} into two parts:
\eqn{
\label{splitintotwo}
\int \prob(D_0>s|W_0=w)\,dF(w)
=
\int_{I_1} \prob(D_0>s|W_0=w)\,dF(w)+
\int_{I_2} \prob(D_0>s|W_0=w)\,dF(w),
}
where $I_1=[0,m(s))$ and $I_2=[m(s),\infty)$, and where
\eqn{
\label{waarde-m}
m(s)=\Big(
\frac{s-s^{1/2}\log s +O(1)}{\xi}
\Big)^{\alpha/d}.
}
As in \cite[Theorem 1.1]{Yukich07}, using Bernstein's inequality, one can show for every $a>0$ that
    \eqn{
    \label{as-prop}
    \lim_{s\to \infty} s^{a} \int_{I_1}\prob(D_0>s|W_0=w)\,dF(w)\le   \lim_{s\to \infty} s^{a} s^{-3\log s/10}=0.
    }
This shows that the integral over $I_1$ does not contribute to the possible regular variation of $\prob(D_0>s)$. In order to investigate the integral over $I_2$, let $Y_w$ denote a random variable with the same distribution as $D_0|W_0=w$, that is,
\eqn{
\label{emyy}
Y_w\stackrel{d}{=}(D_0|W_0=w).
}
The first moment of $Y_w$ is characterized in (\ref{hypothesis}) and a similar analysis as for the first moment yields
$$
\mbox{Var}(Y_w)=\xi'w^{d/\alpha}+O(1),
$$
where $\xi'<\xi$.

The proof of \eqref{degree-tail} is now completed in a slightly different way than in \cite{Yukich07}.
Define
\eqn{
\label{def-G}
G(t)=\int_{w>m(t)} \prob(Y_w>t)\,dF(w),
}
where the function $m$ is defined by \eqref{waarde-m}. Then, \eqref{as-prop} shows that $\prob(D_0>t)=G(t)+O(t^{-a})$ for any $a>0$. Clearly, $\prob(D_0>s)$ is a monotone function on $(0,\infty)$ and hence \eqref{degree-tail} follows if we show that
$$
\lim_{t\to \infty} \frac{\prob(D_0>st)}{\prob(D_0>t)}=s^{-\gamma},
$$
on a dense set $A\subset (0,\infty)$; see \cite[Section VIII.8]{Feller}. By \eqref{splitintotwo} and \eqref{as-prop} this in turn follows if we can deduce that
    $$
    \lim_{t\to \infty} \frac{G(st)}{G(t)}=s^{-\gamma},
    $$
for $s\in (0,\infty)$. To this end, we note that, clearly,
    $$
    G(t)\leq 1-F(m(t)).
    $$
Further, for each $\vep>0$, we have
    $$
    G(t)\geq \int_{w>(1+\vep)m(t)} \prob(Y_w>t)\,dF(w)
    =1-F\big((1+\vep)m(t)\big)+\int_{w>(1+\vep)m(t)} \prob(Y_w\leq t)\,dF(w),
    $$
and, by Chebyshev's inequality and the fact that $\expec[Y_w]>t(1+\vep/2)$ for $t>0$ sufficiently large, we obtain that
    $$
    \prob(Y_w\leq t)\leq \frac{\mbox{Var}(Y_w)}{(\expec[Y_w]-t)^2}\leq C/(t\vep)=o(1)
    $$
uniformly in $w>(1+\vep)m(t)$ as $t\to\infty$. Hence, since
    $$
    \lim_{t\to \infty}\frac{m(ts)}{m(t)}=(s/\xi)^{\alpha/d},
    $$
we arrive at
    $$
    \lim_{t\to \infty} \frac{G(ts)}{G(t)}=\lim_{t\to \infty} \frac{1-F(m(ts))}{1-F(m(t))}
    =\lim_{t\to \infty} \frac{1-F((st/\xi)^{\alpha/d})}{1-F((t/\xi)^{\alpha/d})}
    =s^{-\alpha(\tau-1)/d}.
    $$
\end{proof}

\section{Percolation -- finiteness of the critical value}
\label{sec-perc}
In the following sections, we investigate the percolation properties of our model.
We take $p_{xy}$ as in \eqref{edgeprob} where $\alpha> 0$ is fixed and view $\lambda>0$ as the percolation parameter. When the weights $(W_x)_{x\in\Zbold^d}$ have finite mean, we can,
without loss of generality, assume that they are normalized so that $\expec[W]=1$.

Denote the resulting random graph by $G(\lambda,\alpha)$ and write $x\conn y$ to denote
the event that there is a path of occupied edges between $x$ and $y$ in $G(\lambda,\alpha)$.
Denote by $\mathcal{C}(x)=\{y\colon x\conn y\}$ the \emph{component} of $x$, and
by $|\mathcal{C}(x)|$ the number of vertices in $\mathcal{C}(x)$. The \emph{percolation probability}
is defined as
    $$
    \theta(\lambda)=\prob(|\mathcal{C}(0)|=\infty),
    $$
and the critical percolation value $\lambda_c$ is defined as
    $$
    \lambda_c=\inf\{\lambda\colon \theta(\lambda)>0\}.
    $$
See \cite{BolRio06,Grim99} for general introductions to percolation.
It follows from the general uniqueness result in \cite{GKN} that $G(\lambda,\alpha)$
contains almost surely at most one infinite component. Under what conditions on $\alpha$ and on the
degree distribution is there a non-trivial phase transition in the sense that $\lambda_c\in(0,\infty)$?
Note to begin with that it follows from Proposition \ref{thm:div_deg} that $\lambda_c=0$ for
$\alpha\leq d$ or $\alpha>d$ and $\gamma\leq 1$, that is, when $\alpha\leq d$ or
$\gamma\leq 1$ the graph percolates for all $\lambda>0$.
Hence we shall henceforth restrict to the case $\alpha>d$ and $\gamma> 1$. The following theorem gives sufficient conditions for $\lambda_c<\infty$ so that the model percolates for large enough $\lambda$.

\begin{Theorem}[Finiteness of the critical value]\label{th:finite}
Assume that $\alpha>d$ and that $\gamma>1$.

\item[\rm{(a)}] If $\prob(W=0)<1$, then $\lambda_c<\infty$ in $d\geq 2$.

\item[\rm{(b)}] If $\alpha\in(1,2]$ and $\prob(W\geq w)=1$ for some $w>0$, then $\lambda_c<\infty$ in $d=1$.

\item[\rm{(c)}] If $\alpha>2$ and the weight distribution satisfies
    \eqn{
    \lbeq{power-law-ubd}
    1-F(w)\leq c w^{-(\tau-1)}, \qquad w\ge 0,
    }
for some $c>0$ and $\tau>1$ such that $\gamma=\alpha(\tau-1)/d>2$, then $\lambda_c=\infty$ in $d=1$.
\end{Theorem}

In $d\geq 2$, we hence have $\lambda_c<\infty$ as soon as the weights are not almost surely equal to 0, while
in $d=1$, at least for weights that are bounded away from 0, the behavior is different for
$\alpha\leq 2$ and $\alpha>2$. The above conditions on the weight distribution can presumably be weakened; see Section 6 for some further comments on this.

The proof of part (a) uses the result from \cite{LSS} concerning domination of $r$-dependent random fields by product measures, where a random field $(X_z)_{z\in\mathbb{Z}^d}$ is said to be $r$-dependent if for any two sets $A,B\subset \mathbb{Z}^d$ at $l_\infty$-distance at least $r$ from each other we have that $(X_z)_{z\in A}$ is independent of $(X_z)_{z\in B}$. The version we need is as follows.

\begin{Theorem}[Liggett, Schonmann $\&$ Stacey (1997)]\label{th:LSS}
For each $d\geq 2$ and $r\geq 1$ there exists a $p_c=p_c(d,r)<1$ such that the following holds.
For any $r$-dependent random field $(X_z)_{z\in\mathbb{Z}^d}$ satisfying
$\prob(X_z=1)=1-\prob(X_z=0)\geq p$, with $p>p_c$, the 1's in
$(X_z)_{z\in\mathbb{Z}^d}$ percolate almost surely.
\end{Theorem}

Theorem \ref{th:LSS} is formulated in terms of \emph{sites} rather than \emph{edges} or
\emph{bonds}. It is a classical result that any bond percolation model can be formulated in terms
of a site percolation model, see e.g.\ \cite{Grim99}.

\begin{proof}[Proof of Theorem \ref{th:finite}]
We begin with (a). Say that a vertex $x\in\mathbb{Z}^d$ is $\vep$-good if $W_x\geq \vep$ and note that, if two nearest-neighbor sites $x$ and $y$ are both $\vep$-good, then the probability that the edge between them is occupied in $G(\lambda,\alpha)$ is at least $1-\e^{-\lambda \vep^2}$. Hence it suffices to show that the edge configuration obtained by independently keeping edges between  $\vep$-good nearest-neighbor vertices with probability $1-\e^{-\lambda \vep^2}$ and removing all other edges percolates for some $\vep>0$. To this end, say that a vertex $z\in\mathbb{Z}^d$ is $\vep$-open if all the $2d$ edges to its nearest-neighbors are present in this configuration and let
$X_z=1$ precisely when $z$ is $\vep$-open. Note that this defines a 3-dependent random field and that
$$
\prob(X_z=1)=\prob(z \mbox{ is $\vep$-open})\geq \prob(W\geq \vep)^{2d+1}(1-\e^{-\lambda \vep^2})^{2d}.
$$
By the assumption that $\prob(W=0)<1$, the first factor can be made arbitrarily close to 1 by picking $\vep$ suitably small and the second factor can then be made arbitrarily close to 1 by taking $\lambda$ large. Hence, if $\vep$ is small enough then, by Theorem \ref{th:LSS}, we can make
$\prob(X_z=1)$ large enough to guarantee that the $\vep$-open vertices percolates.

Part (b) is a direct consequence of the results proved in \cite{NewSch}: If $\prob(W\geq \vep)=1$, then clearly the edge configuration stochastically dominates a configuration with independent edges and $p_{xy}=1-\e^{-\lambda \vep^2/|x-y|^\alpha}$. It is shown in \cite{NewSch} that, for $\alpha\in(1,2]$, this model percolates in $d=1$ for large $\lambda$.

As for (c), we adapt the argument in \cite{Sch}. We start by giving the
proof when $\expec[W]<\infty$. For $x\in\mathbb{Z}$,
let $A_x$ be the event that no vertex $y\leq x$ is connected to any vertex $z>x$.
The sequence $(\mathbf{1}_{A_x})_{x\in\mathbb{Z}}$ is stationary with common mean $\prob(A_0)$.
For $n\geq 1$, write $A_0^{(n)}$ for the event that none of the $n$ edges $(0,n),(-1,n-1),\ldots,(-n+1,1)$
is present in the graph. By the conditional independence of the edges given $(W_x)_{x\in \mathbb{Z}^d}$, we have that
\begin{eqnarray*}
\prob(A_0) & = &
\E\left[\prod_{n=1}^\infty\prob\left(A_0^{(n)}|(W_x)_{x\in\mathbb{Z}}\right)\right] \\
& = & \E\left[\exp\left\{-\sum_{n=1}^\infty\frac{\lambda}{n^\alpha}(W_0W_n+\ldots+W_{-n+1}W_1)\right\}
\right].
\end{eqnarray*}
Since $\e^{-x}$ is a convex function, it follows from Jensen's inequality and the fact that $(W_x)_{x\in\Zbold}$ are i.i.d.\ with mean 1 that
$$
\prob(A_0)\geq \exp\left\{-\sum_{n=1}^\infty\frac{\lambda}{n^\alpha}\E[W_0W_n+\ldots+W_{-n+1}W_1]\right\}=
\exp\left\{-\lambda\sum_{n=1}^\infty \frac{1}{n^{\alpha-1}}\right\}.
$$
Hence, $\prob(A_0)>0$ when $\alpha>2$. The ergodic theorem applied to the sequence
$(\mathbf{1}_{A_x})_{x\in\mathbb{Z}}$ then gives that infinitely many of the $A_x$'s occur almost surely, implying that all components are finite. This completes the proof of part (c) when $\expec[W]<\infty$.

In the general case, we have that
    \eqn{
    \prob(A_0)=\E\left[\exp\left\{-\lambda\sum_{i,j\geq 0\colon (i,j)\neq (0,0)}^{\infty}\frac{W_{-i}W_j}{(j+i)^{\alpha}}\right\}\right]>0
    }
precisely when the double sum is finite a.s.
We bound
    \eqn{
    \sum_{i,j\geq 0\colon(i,j)\neq (0,0)}^{\infty}\frac{W_{-i}W_j}{(j+i)^{\alpha}}
    \leq Z_1Z_2,
    }
where
    \eqn{
    Z_1=\sum_{j=0}^{\infty}\frac{W_{j}}{(j\vee 1)^{\alpha/2}},
    \qquad
    Z_2=\sum_{i=0}^{\infty}\frac{W_{-i}}{(i\vee 1)^{\alpha/2}}.
    }
The random variables $Z_1$ and $Z_2$ have the same distribution, so that we only need to check that
$Z_1<\infty$ a.s. Then, the remainder of the proof
can be completed as in the case where $\expec[W]<\infty$.

We continue to prove that $Z_1<\infty$ a.s.\ when $\gamma>2$. Take $a_i=i^{(1+\vep)/(\tau-1)}$ for some $\vep>0$. Then, the events $\{W_i>a_i\}$ occur only finitely often, since
    \eqn{
    \prob(W_i>a_i)
    =[1-F](a_i)\leq c a_i^{-(\tau-1)}=c i^{-(1+\vep)},
    }
which is summable in $i$. Then, we split
$Z_1=Y_1+Y_2$, where
    \eqn{
    Y_1\equiv \sum_{j=0}^{\infty}\frac{(W_j\vee a_j)}{(j\vee 1)^{\alpha/2}},
    \qquad
    Y_2\equiv \sum_{j=1}^{\infty} \frac{(W_j-a_j)\indic{W_j>a_j}}{(j\vee 1)^{\alpha/2}}.
    }
The sum in the definition of $Y_2$ contains only finitely many terms a.s.,
and is thus finite a.s. Further, note that
    \eqn{
    \expec[(W_j\wedge x)]\leq \sum_{y=1}^x [1-F](y)\leq c \sum_{y=1}^x y^{-(\tau-1)}
    \leq cy^{2-\tau}.
    }
and hence
    \eqn{
    \frac{\expec[(W_j\wedge a_j)]}{(j\vee 1)^{\alpha/2}}
    \leq c a_j^{2-\tau} (j\vee 1)^{-\alpha/2}
    \leq c (j\vee 1)^{-\alpha/2 + (1+\vep)(2-\tau)/(\tau-1)}.
    }
When $\gamma=\alpha(\tau-1)>2$, we have that $-\alpha/2 +(2-\tau)/(\tau-1)<-1$.
Therefore, we can take $\vep>0$ so small
that $-\alpha/2 + (1+\vep)(2-\tau)/(\tau-1)<-1,$ which makes $Y_1$ have finite mean. In particular, it is finite a.s.
\end{proof}

\section{Percolation -- positivity of the critical value}
\label{sec-crit-val-pos}
\label{sec-pos}

In this section we show that $\lambda_c>0$ if and only if the degrees have finite variance. We give the proof in two subsections. First we show that there is no percolation for small $\lambda$ when the degrees have finite variance and then that there is percolation for all $\lambda>0$ when the variance of the degrees is infinite.

\subsection{The critical value is positive for finite-variance degrees}

Recall from Section \ref{sec-vert} that, if the weight tail $\prob(W>w)$ varies regularly with exponent $\tau-1$, then the degree tail $\prob(D_0>s)$ varies regularly with exponent $\gamma=\alpha(\tau-1)/d$. In this section we show that $\lambda_c>0$ when $\gamma>2$, that is, when the degrees have finite variance. Recall that we assume throughout that $\alpha>d$. As pointed out after Theorem \ref{th:deg}, finite variance for the weights then implies finite variance for the degrees, but the degrees may have finite variance also if the weights have infinite variance. We first prove that $\lambda_c>0$ in the case when the weights have finite variance. The proof is slightly simpler in that case and gives an explicit lower bound for $\lambda_c$. We then extend the arguments to cover also the case when $\gamma>2$ but $\tau\in(2,3)$.

\begin{Theorem}[Positivity of the critical value for finite-variance weights]
\label{thm-crit-value-pos-finvar}
Assume that $\expec[W^2]<\infty$. Then, $\theta(\lambda)=0$ for every $\lambda<1/(\expec[W^2]\sum_{x\neq 0} |x|^{-\alpha})$, that is,
    $$
    \lambda_c\geq 1/\Big(\expec[W^2]\sum_{x\neq 0} |x|^{-\alpha}\Big).
    $$
\end{Theorem}

\proof Since $\expec[W]=1<\infty$, every vertex has a.s.\ bounded degree.
As a result, the event $\{|{\cal C}(0)|=\infty\}$ implies that, for every $n\geq 1$,
there is a path of distinct occupied edges of length at least $n$ starting from the origin. Thus,
    \begin{eqnarray*}
    \theta(\lambda)&=\prob(|{\cal C}(0)|=\infty)
    \leq \sum_{(x_1,\ldots, x_n)} \prob((0,x_1), (x_1,x_2),\ldots, (x_{n-1},x_n)\text{ occupied})\\
    &=\sum_{(x_1,\ldots, x_n)} \expec\Big[\prob\Big((0,x_1), (x_1,x_2),\ldots, (x_{n-1},x_n)\text{ occupied}\mid (W_x)_{x\in\Zbold^d}\Big)\Big],
    \nn
    \end{eqnarray*}
where the sum is over $(x_1,\ldots, x_n)\in \big({\mathbb Z}^d\big)^n$ such that every vertex occurs at most once in the path $(0,x_1,\ldots, x_n)$. We call such paths \emph{self-avoiding paths}. By the conditional independence of the edges given the weights, we have that
    $$
    \prob\Big((0,x_1), (x_1,x_2),\ldots, (x_{n-1},x_n)\text{ occupied}\mid (W_x)_{x\in\Zbold^d})
    =\prod_{i=1}^n p_{x_{i-1},x_i},
    $$
where $p_{xy}$ is defined in \refeq{edgeprob} and, by convention, $x_0=0$. Therefore,
    $$
    \lbeq{theta-ub-1}
    \theta(\lambda)
    \leq \sum_{(x_1,\ldots, x_n)} \expec\Big[\prod_{i=1}^n p_{x_{i-1},x_i}\Big].
    $$
We next use the fact that $1-\e^{-x}\leq x$ to conclude that
    \eqn{
    \label{ineq-pxy}
    p_{x,y}\leq \lambda W_xW_y/|x-y|^{\alpha}.
    }
It follows that
    \begin{eqnarray*}
    \theta(\lambda)
    &\leq & \sum_{(x_1,\ldots, x_n)} \expec\Big[\prod_{i=1}^n \lambda W_{x_{i-1}}W_{x_i}/|x_{i-1}-x_i|^{\alpha}\Big]\\
    & =&\lambda^n \sum_{(x_1,\ldots, x_n)}\Big(\prod_{i=1}^n \frac{1}{|x_{i-1}-x_i|^{\alpha}}\Big)\expec\Big[W_0W_{x_n} \prod_{i=1}^{n-1}W_{x_i}^2\Big].\nn
    \end{eqnarray*}
Since every vertex occurs at most once in the path $(0,x_1,\ldots, x_n)$ and $(W_x)_{x\in\Zbold^d}$
are i.i.d.\ with mean 1, we have that
    $$
    \expec\Big[W_0W_{x_n} \prod_{i=1}^{n-1}W_{x_i}^2\Big]
    =\expec[W]^2\expec[W^2]^{n-1}.
    $$
Further, by translation invariance,
    $$
    \sum_{(x_1,\ldots, x_n)}\prod_{i=1}^n \frac{1}{|x_{i-1}-x_i|^{\alpha}}
    \leq \Big(\sum_{x\neq 0} \frac{1}{|x|^{\alpha}}\Big)^n.
    $$
As a result,
    $$
    \theta(\lambda)
    \leq \frac{\expec[W]^2}{\expec[W^2]} \Big(\lambda \expec[W^2]\sum_{x\neq 0} \frac{1}{|x|^{\alpha}}\Big)^n.
    $$
Thus, when $\lambda<1/(\expec[W^2]\sum_{x\neq 0} |x|^{-\alpha})$, the right hand side converges to 0 as $n\rightarrow \infty$, which implies that $\theta(\lambda)=0$.
\qed
\medskip

We relax the assumption in Theorem \ref{thm-crit-value-pos-finvar} to finite variance for the degrees, that is, $\gamma>2$.

\begin{Theorem}[Positivity of the critical value for finite-variance degrees]
\label{thm-crit-value-pos}
Assume that there exists $\tau>1$ and $c>0$ such that
    \eqn{
    \lbeq{[1-F]-infvar-ub}
    [1-F](x)=
    \prob(W>x)\leq c x^{-(\tau-1)},\qquad x\ge 0,
    }
with $\gamma=\alpha(\tau-1)/d>2$. Then, $\theta(\lambda)=0$ for small $\lambda>0$, that is, $\lambda_c>0$.
\end{Theorem}

\proof  The proof is an adaptation of the proof of Theorem \ref{thm-crit-value-pos-finvar}. Instead of \eqref{ineq-pxy} we use the bound
    $$
    p_{x,y}\leq \Big(\lambda W_xW_y/|x-y|^{\alpha}\wedge 1\Big).
    $$
As a result,
    \begin{eqnarray*}
    \theta(\lambda)
    &\leq \sum_{(x_1,\ldots, x_n)} \expec\Big[\prod_{i=1}^n
    \Big(\lambda W_{x_{i-1}}W_{x_i}/|x_{i-1}-x_i|^{\alpha}\wedge 1\Big)\Big].
    \end{eqnarray*}
By Cauchy-Schwarz's inequality and the independence of $(W_{x})_{x\in {\mathbb Z}^d}$ we obtain
    $$\begin{array}{lll}
    \expec\Big[\prod_{i=1}^n
    \Big(\lambda W_{x_{i-1}}W_{x_i}/|x_{i-1}-x_i|^{\alpha}\wedge 1\Big)\Big]^2\\
    \leq
    \expec\Big[\prod_{i=1}^{\lceil n/2\rceil}
    \Big(\lambda W_{x_{2i-1}}W_{x_{2i}}/|x_{2i-1}-x_{2i}|^{\alpha}\wedge 1\Big)^2\Big]\nn\\
    \quad \times\expec\Big[\prod_{i=1}^{\lfloor n/2 \rfloor}
    \Big(\lambda W_{x_{2i-2}}W_{x_{2i-1}}/|x_{2i-2}-x_{2i-1}|^{\alpha}\wedge 1\Big)^2\Big]\nn\\
    =\prod_{i=1}^n\expec\Big[\Big(\lambda W_{x_{i-1}}W_{x_i}/|x_{i-1}-x_{i}|^{\alpha}\wedge 1\Big)^2\Big].\nn
    \end{array}$$
Therefore,
    $$
    \theta(\lambda)
    \leq \sum_{(x_1,\ldots, x_n)}\prod_{i=1}^n
    g(|x_{i-1}-x_i|^{\alpha}/\lambda)^{1/2},
    $$
where we define
    $$
    g(u)=\expec\Big[\Big(W_1W_2/u\wedge 1\Big)^2\Big].
    $$
By translation invariance,
    $$
    \theta(\lambda)
    \leq \Big(\sum_{x\neq 0} g\Big(|x|^{\alpha}/\lambda\Big)^{1/2}\Big)^n.
    $$
We continue by investigating the asymptotics of $u\mapsto g(u)$ for $u\rightarrow \infty$:

\begin{Lemma}[Asymptotics of $g$]
\label{lem-g-bd}
When the distribution function $F$ satisfies \refeq{[1-F]-infvar-ub} for some $\tau>1$, then
there exists a constant $C>0$ such that
    $$
    g(u)\leq C (1+\log{u}) u^{-\big((\tau-1)\wedge 2\big)}.
    $$
\end{Lemma}

\proof
Let $H$ denote the distribution function of $W_1W_2$, where $W_1$ and $W_2$ are two independent copies of a
random variable $W$ with distribution function $F$. When $F$ satisfies \refeq{[1-F]-infvar-ub}, then it is not hard to see that there exists a $C>0$ such that
    \eqn{
    \lbeq{1-G-bd}
    [1-H](u)\leq C (1+\log{u}) u^{-(\tau-1)}.
    }
Indeed, assume that $F$ has a density $f(w)=c w^{-\tau}$, for $w\geq 1$.
Then
    $$
    [1-H](u)
    =\int_1^{\infty} f(w) [1-F](u/w)dw.
    $$
Clearly, $[1-F](w)=c' w^{-(\tau-1)}$ for $w\geq 1$ and $[1-F](w)=1$ otherwise. Substitution of this
yields
    $$
    [1-H](u)
    =cc' \int_1^{u} w^{-\tau} (u/w)^{-(\tau-1)}dw+c\int_u^{\infty}
    w^{-\tau}
    \leq C (1+\log{u}) u^{-(\tau-1)}.
    $$
When $F$ satisfies \refeq{[1-F]-infvar-ub}, then $W_1$ and $W_2$ are stochastically upper bounded by $W^*_1$ and $W^*_2$ with distribution function $F^*$ satisfying $[1-F^*](w)=cw^{-(\tau-1)}$, and the claim in \refeq{1-G-bd} follows from the above computation.

Let $V=W_1W_2$, so that $V$ has distribution function $H$. We complete the proof of Lemma \ref{lem-g-bd} by bounding
    $$
    \expec\Big[\big(V/u\wedge 1\big)^2\Big]
    \leq 1-H(u) +u^{-2}\int_1^u 2v [1-H(v)]dv
    \leq C (\log{u}+1) u^{-\big[(\tau-1)\wedge 2\big]},
    $$
where, in the last inequality, we distinguish between the case $\tau>3$, in which $\int_0^\infty 2v [1-H(v)]dv<\infty$, and $\tau\in (1,3]$.
\qed
\medskip

Lemma \ref{lem-g-bd} yields
    \begin{eqnarray*}
    \theta(\lambda) &
    \leq & \Big(C\lambda^{(\tau-1)/2}
    \sum_{x\neq 0} \Big(\log(|x|^\alpha/\lambda)+1\Big)^{1/2}|x|^{-\alpha\big[(\tau-1)\wedge 2\big]}\Big)^n\\
    & \leq & \Big(C\lambda^{(\tau-1)/4}\sum_{x\neq 0} \Big(\log|x|^\alpha+1\Big)^{1/2}|x|^{-\alpha\big[(\tau-1)\wedge 2\big]}/2\Big)^n
    \end{eqnarray*}
where the last inequality holds when $\lambda<1$ is small enough to ensure that $\lambda^{(\tau-1)/2}(1-\log\lambda)\leq 1$. Now, $\sum_{x\neq 0} (\log{|x|^\alpha}+1)^{1/2} |x|^{-\alpha\big[(\tau-1)\wedge 2\big]/2}<\infty$ when $\alpha[(\tau-1)\wedge 2]/d>1$.
Since $\alpha >d$, for $\tau<3$ we find as condition that $\gamma=\alpha(\tau-1)/d>2$. Thus, when $\lambda$ is also so small that
    $$
    C\lambda^{(\tau-1)/4}
    \sum_{x\neq 0} (\log{|x|^\alpha}+1)^{1/2}|x|^{-\alpha(\tau-1)/2}
    <1,
    $$
we have that $\theta(\lambda)=0$. This completes the proof of Theorem \ref{thm-crit-value-pos}.
\qed

\subsection{The critical value is zero for infinite-variance degrees}

In this section we will investigate the case when $\gamma<2,$ so that the degrees have infinite variance.
As we will show, the critical value equals zero in this case.

\begin{Theorem}[Critical value equals zero for infinite-variance degrees]
\label{thm-crit-value-zero}
Assume the existence of $\tau>1$ and $c>0$ such that the weight distribution $F$ satisfies
    \eqn{
    \lbeq{[1-F]-infvar-lb}
    [1-F](x)=
    \prob(W>x)\geq c x^{-(\tau-1)}\qquad x\ge 0.
  }
Furthermore, assume that $\alpha>d$ and $\gamma=\alpha(\tau-1)/d<2$.
Then, $\theta(\lambda)>0$ for every $\lambda>0$, that is, $\lambda_c=0$.
\end{Theorem}

\proof Clearly, it suffices to show that $\theta(\lambda)>0$ as soon as $\lambda>0$. To this end, take a radius $r_{\lambda}=\lfloor \lambda^{-q}\rfloor$ for some $q>0$ to be determined later on.
Let $B(x,r)=\{y\colon |y-x|\leq r\}$ denote the Euclidean ball of radius $r$ around $x$,
and write $B(r)=B(0,r)$. Furthermore, for $x\in \Zbold^d$, define
    $$
    M_x(\lambda)=r_{\lambda}^{-d/(\tau-1)}\max_{y\in \Zbold^d\cap B(r_{\lambda}x, r_\lambda)} W_y.
    $$
Then, for small $\lambda>0$, we have
$$
\prob(M_x(\lambda)\geq \vep)=1-F\big(\vep r_{\lambda}^{d/(\tau-1)}\big)^{(2r_{\lambda}+1)^d}
\geq 1-\big(1-c\vep^{-(\tau-1)} r_{\lambda}^{d}\big)^{(2r_{\lambda}+1)^d}
\geq 1-\e^{-c\vep^{1-\tau}}.
$$
Hence $\prob(M_x(\lambda)\geq \vep)\to 1$ uniformly in $\lambda$ as $\vep\to 0$. Say that $x\in \Zbold^d$ is \emph{good} when $M_{x}(\lambda)\geq \vep$. The events that different sites are good are independent and have the same probability. Given two nearest-neighbors $x,y\in \Zbold^d$, we say that $(x,y)$ is $\lambda$-occupied when there is a direct edge between $x(\lambda)$ and $y(\lambda)$, where $x(\lambda)$ is the vertex that maximizes $W_z$ for $z\in B(r_{\lambda}x,r_{\lambda})$ and $y(\lambda)$ the vertex that maximizes $W_z$ for $z\in B(r_{\lambda}y,r_{\lambda})$. Then, when $x,y$ are both good,
    \begin{eqnarray*}
    \prob((x,y) \text{ $\lambda$-occupied}\mid x,y \text{ good})
    & =& \expec\Big[p_{x(\lambda),y(\lambda)}\mid x,y \text{ good}\Big]\\
    & = & \expec\Big[1-\e^{-\lambda W_{x(\lambda)}W_{y(\lambda)}/|x(\lambda)-y(\lambda)|^\alpha}\mid x,y \text{ good}\Big]\nn\\
    &\geq & 1-\e^{-\lambda (\vep r_{\lambda}^{d/(\tau-1)})^2/|x(\lambda)-y(\lambda)|^\alpha}\nn\\
    &\geq & 1-\e^{-\lambda \vep^2 r_{\lambda}^{2d/(\tau-1)-\alpha}}.\nn
    \end{eqnarray*}
Recall that $r_{\lambda}=\lfloor \lambda^{-q}\rfloor$. Hence,
    $$
    \lambda r_{\lambda}^{2d/(\tau-1)-\alpha}
    \geq \lambda^{1-q(2d/(\tau-1)-\alpha)}.
    $$
By assumption, $\gamma<2$, so that $\alpha(\tau-1)<2d$. Thus, we can take $q>1/(2d/(\tau-1)-\alpha)$, so that $\lambda^{1-q(2d/(\tau-1)-\alpha)} \rightarrow \infty$ when $\lambda\downarrow 0$. Then, for every $\vep>0$, we have
    $$
    \lim_{\lambda\downarrow 0}
    \prob((x,y) \text{ $\lambda$-occupied}\mid x,y \text{ good})=1.
    $$
The limit also holds for several edges simultaneously.

We now define a nearest-neighbor bond percolation model on $\Zbold^d$, where the bond between nearest-neighbor sites $x,y\in \Zbold^d$ is open when both $x$ and $y$ are good and there is a direct edge between $x(\lambda)$ and $y(\lambda)$, that is, when $(x,y)$ is $\lambda$-occupied.
The probability that a vertex is good can be made as close to 1 as we wish by taking $\vep>0$ small, and the edge probability can then be made as close to 1 as we like by taking $\lambda$
sufficiently small. Thus, by applying Theorem \ref{th:LSS} as in the proof of Theorem \ref{th:finite}(a), we conclude that the model will percolate with probability 1 when $\vep$ and $\lambda$ are sufficiently small. Denote by $\theta(\lambda,\vep)$ the probability that 0 percolates in the above bond percolation model. Note that 0 percolates in our original inhomogeneous model when (a) 0 percolates in the bond model; and (b) $0$ is directly connected to $0(\lambda)$, which is the vertex that maximizes $W_z$ in $B(0,r_{\lambda})$. Now, the probability that $0$ is directly connected to $0(\lambda)$, conditionally on $0$ being good, is at least
$\prob(W\geq \vep)[1-\e^{-\vep^2 \lambda r_{\lambda}^{d/(\tau-1)-\alpha}}]$.
Therefore,
    $$
    \theta(\lambda)\geq
    \prob(W\geq \vep)\big[1-\e^{-\vep^2 \lambda r_{\lambda}^{d/(\tau-1)-\alpha}}\big]
    \theta(\lambda,\vep)
    >0.
    $$
As a result, $\lambda_c\leq \lambda$, which holds for every $\lambda>0$, so that $\lambda_c=0$.
\qed
\medskip

\section{Distances}
\label{sec-dist}
For $x\in\Zbold^d$, write $d(0,x)$ for the graph distance between 0 and $x$, that is, $d(0,x)$ is the minimum number of edges that form a path from 0 to $x$. If 0 and $x$ are not connected, then we define $d(0,x)=\infty$. In this section we show that, conditionally on $0$ and $x$ being connected, the distance $d(0,x)$ is of order $\log\log|x|$ as $|x|\to\infty$ when the degrees have infinite variance and $\lambda>\lambda_c$. When the degrees have a finite variance on the other hand, then $d(0,x)$ is at least of order $\log|x|$ when $\alpha>d$ and at least of order $|x|^{\vep}$ for some $\vep>0$ when $\alpha>2d$.

\subsection{Doubly logarithmic asymptotics for infinite-variance degrees}

We start by proving a loglog upper bound when the degrees have infinite variance.
Recall that $x\conn y$ denotes the event that $x$ and $y$ are in the same component.

\begin{Theorem}[Doubly logarithmic upper bound on distances for infinite-variance degrees]
\label{thm-log-dist-infvar}
Assume that there exists $\tau>1$ and $c>0$
such that \refeq{[1-F]-infvar-lb} holds and such that
$\gamma=\alpha(\tau-1)/d\in (1,2)$. Finally, assume that $\lambda>0$.
Then, for every $\eta>0$, we have
    $$
    \lim_{|x|\rightarrow \infty} \prob\Big(d(0,x)\leq (1+\eta)\frac{2\log\log{|x|}}{|\log(\gamma-1)|}\Big|\,0\conn x\Big)=1.
    $$
\end{Theorem}

\proof
Throughout this proof $c_1,c_2,\ldots$ denote strictly positive constants. Let $(W_i)_{i=1}^n$ be an i.i.d.\ collection of weight variables. When the weight distribution $F$ satisfies \eqref{[1-F]-infvar-lb}, it is easy to see that, for any $\delta\in(0,1)$, we can bound
    \begin{equation}\label{eq:maxw_bd}
    \prob\Big(\max_{1\leq i\leq n}W_i\leq n^{(1-\delta)/(\tau-1)}\Big)\leq \left(1-\frac{c}{n^{1-\delta}}\right)^n\leq \e^{-cn^\delta}.
    \end{equation}
Take $x\in\Zbold^d$ with $|x|$ large and let $b\in(0,1)$ be a constant whose value will be specified later. For $i=0,1,2\ldots$, write $\widetilde{B}(x,b^i)$ for the ball with radius $|x|^{b^i}/4$ centered at the point at distance $|x|^{b^i}/2$ from 0 on the line segment from 0 to $x$. Furthermore, let $z_i\in\Zbold^d$ be the (random) vertex in $\widetilde{B}(x,b^i)$ with maximal weight. We have
    $$
    \prob\left(\cup_{i=0}^{k-1} \{(z_i,z_{i+1})\mbox{ not occupied}\}\right)\leq \sum_{i=0}^{k-1}\E\left[\e^{-\lambda W_{z_i}W_{z_{i+1}}/|z_i-z_{i+1}|^\alpha}\right].
    $$
Using (\ref{eq:maxw_bd}) and the fact that the number of points in $\Zbold^d \cap \widetilde{B}(x,b^i)$ is of the order $|x|^{db^i}$, this can be bounded by
$$
\sum_{i=0}^{k-1}\left(\exp\left\{-\frac{c_1|x|^{db^i(\tau-1)^{-1}(1-\delta)}|x|^{db^{i+1}(\tau-1)^{-1}(1-\delta)}}{|x|^{\alpha b^i}}\right\}+\exp\left\{-c|x|^{\delta db^i}\right\} \right),
$$
where the exponent in the first term simplifies to
$$
-c_1|x|^{b^i[d(\tau-1)^{-1}(1-\delta)(1+b)-\alpha]}.
$$
Fix $b\in(\gamma-1,1)$, and then take $\delta$ small so that $d(\tau-1)^{-1}(1-\delta)(1+b)-\alpha>0$.
Using the fact that $\sum_{i=1}^{k-1}\e^{-|x|^{b^i}}
=\Theta\big(\e^{-|x|^{b^{k-1}}}\big)$ as $k\to\infty$, we can hence bound
$$
\prob\Big(\cup_{i=0}^{k-1} \{(z_i,z_{i+1})\mbox{ not occupied}\}\Big)\leq c_2\exp\left\{-c_3\big(|x|^{b^{k-1}}\big)^{c_4}\right\}
$$
for large $k$.

Now fix $\vep>0$. Take $A=A(\vep)$ large so that $c_2\e^{-c_3A^{c_4}}\leq \vep$ and then choose $k$ such that $|x|^{b^k}= A$, that is,
$$
k=\frac{\log\log|x|+\log\log A}{|\log b|}.
$$
Then $\prob\big((z_i,z_{i+1})\mbox{ occupied for all }i=1,\ldots,k-1 \big)\geq 1-\vep$ and, on the event that all $(z_i,z_{i+1})$ are occupied, we have that
$$
d(0,x)\leq k+d(0,z_k),
$$
where $|z_k|\leq|x|^{b^k}\leq A$. We now let $|x|\to \infty$ and use the fact that the model has a {\it unique} infinite cluster. If $0\conn x$ and all $(z_i,z_{i+1})$ are occupied, then $0\conn z_k$, which in turn implies that $d(0,z_k)<\infty$ almost surely (since $|z_k|\leq A$). Hence, on the event that $0\conn x$, for any $\kappa>0$, we have that $\prob(d(0,z_k)\leq \kappa\log\log|x|)\geq 1-\vep$ when $k$ is large. It follows that
$$
\lim_{|x|\to\infty}\prob\left(d(0,x)\leq \frac{2(1+\eta/2)\log\log|x|}{|\log b|}\,\Big|\,0\conn x\right)\geq 1-2\vep.
$$
The proof is completed by taking $b$ close enough to $\gamma-1$ to ensure that
    $$
    \frac{1+\eta/2}{|\log b|}\leq \frac{1+\eta}{|\log(\gamma-1)|}.
    $$
\qed
\medskip

We continue by proving that, for $\gamma\in (1,2)$, typical distances really are of the order
$\log\log{|x|}$.

\begin{Theorem}[Doubly logarithmic lower bound on distances for infinite-variance degrees]
\label{thm-log-dist-infvar-low}
Assume that there exists $\tau>1$ and $c>0$
such that \refeq{[1-F]-infvar-lb} holds and such that
$\gamma=\alpha(\tau-1)/d\in (1,2)$. Finally, assume that $\lambda>0$.
Then, for every $\eta>0$,
    $$
    \lim_{|x|\rightarrow \infty} \prob\Big(d(0,x)\geq (1-\eta)\frac{2\log\log{|x|}}{|\log(\kappa)|}\Big)=1,
    $$
where $\kappa=\gamma-1$ when $\tau\in (1,2]$ and $\kappa=\alpha/d-1$ when $\tau>2$.
\end{Theorem}

When $\tau\in (1,2]$, we see that the constants in Theorems \ref{thm-log-dist-infvar}--\ref{thm-log-dist-infvar-low}
agree, so that we obtain convergence in probability:

\begin{Corollary}[Doubly logarithmic distances for infinite-variance degrees and infinite-mean weights]
\label{cor-loglog-dist}
Assume that there exists $\tau\in(1,2]$ and $c>0$
such that \refeq{weight-distr} holds and such that
$\gamma=\alpha(\tau-1)/d\in (1,2)$. Finally, assume that $\lambda>0$.
Then, conditionally on $0\conn x$, we have for every $\eta>0$ that
    $$
    \frac{d(0,x)}{\log\log{|x|}}\convp \frac{2}{|\log(\gamma-1)|}.
    $$
\end{Corollary}

\proof By Potter's Theorem (\cite{BGT}), for every $\vep>0$, there exist constants $c_{\vep}$ and $C_{\vep}$ such that
for all $x$ sufficiently large,
    \eqn{
    c_{\vep} w^{-(\tau-1+\vep)}\leq [1-F](w)\leq C_{\vep} w^{-(\tau-1-\vep)}.
    }
The result follows from Theorems \ref{thm-log-dist-infvar}--\ref{thm-log-dist-infvar-low} and
the fact that $\kappa=\gamma-1$ for $\tau\in (1,2]$.
\qed

\noindent
{\it Proof of Theorem \ref{thm-log-dist-infvar-low}.} Define
    \eqn{
    S_n(x)=\sup\{|x-y|\colon d(x,y)\leq n\},
    }
to be the distance between $x$ and the furthest point $y\in \mathbb{Z}^d$ that can be reached via
at most $n$ edges. Then, clearly,
    \eqn{
    \lbeq{bd-d-Sn}
    \prob(d(0,x)\leq 2n)
    \leq \prob(S_n(x)\geq |x|/2)+\prob(S_n(0)\geq |x|/2)
    =2\prob(S_n(0)\geq |x|/2),
    }
where the last equality follows from translation invariance. Now,
for any $s\leq t$, we obtain the recursive bound
    \eqn{
    \lbeq{recur-Sn}
    \prob(S_n(0)\geq t)\leq \prob(S_{n-1}(0)\geq s)+
    \prob(S_{n-1}(0)<s, S_n(0)\geq t),
    }
where, using Boole's inequality, we can further bound
    \eqan{
    \lbeq{Sn-recur-bd}
    \prob(S_{n-1}(0)<s, S_n(0)\geq t)
    &=
    \prob(\exists u,v, \,\mbox{satisfying: $|u|\leq s$ and $|v|\geq t$ such that}\,\, u \conn v)\nn\\
    &\leq
    \sum_{u,v\colon |u|\leq s, |v|\geq t} \expec[p_{u,v}]
    \leq \sum_{u,v\colon |u|\leq s, |v|\geq t} \expec\Big[\Big(\frac{\lambda W_uW_v}{|u-v|^{\alpha}}\wedge 1\Big)\Big]
    \nn\\
    &\leq \sum_{u,v\colon |u|\leq s, |v|\geq t} g_1\Big(\frac{|u-v|^{\alpha}}{\lambda}\Big),
    }
where
    \eqn{
    g_1(u)=\expec\Big[\Big(\frac{W_1W_2}{u}\wedge 1\Big)\Big].
    }
It follows quite easily from the statement in Lemma \ref{lem-g-bd}, that there exists a
constant $C>0$ such that
    \eqn{
    \lbeq{g1-bd}
    g_1(u)\leq C (1+\log{u}) u^{-((\tau-1)\wedge 1)}.
    }
By a computation, similar to the one in the proof of Proposition \ref{prop:exp_deg_bd}
we find that
\begin{eqnarray*}
    \prob(S_{n-1}(0)<s, S_n(0)\geq t)
    &\leq& C\sum_{u,v\colon |u|\leq s, |v|\geq t}  |u-v|^{-\alpha((\tau-1)\wedge 1)}
    (1+\log{|u-v|^{\alpha}/\lambda}) \\
    &\leq& K |s|^d |t|^{d-\alpha [(\tau-1)\wedge 1]+\eta},
\end{eqnarray*}
where $\eta>0$ may be taken arbitrarily small and compensates the $\log$-term. Recall the
definition of $\kappa$ in the theorem.
Fix $A\geq 1$ to be large, and take $\delta>0$ so  that $\kappa-\delta\in (0,1)$.
Then take $t=A^{(\kappa-\delta)^{-n}}$ and $s=A^{(\kappa-\delta)^{-(n-1)}}$, so that $s=t^{\kappa-\delta}$, and
    \eqn{
    K |s|^d |t|^{d-\alpha [(\tau-1)\wedge 1]+\eta}
    =K |t|^{d-\alpha [(\tau-1)\wedge 1]+(\kappa-\delta)d}
    =K \Big(A^{(\kappa-\delta)^{-n}}\Big)^{-\zeta},
    }
with
    \eqn{
    \zeta=\alpha [(\tau-1)\wedge 1]-d-(\kappa-\delta)d-\eta>0,
    }
since $\kappa=\alpha[(\tau-1)\wedge 1]/d-1$.
Combining \eqref{recur-Sn} with these bounds yields the explicit recursion
\eqn{
\prob\left(S_n(0)\geq A^{(\kappa-\delta)^{-n}}\right)
\leq \prob\left(S_{n-1}(0)\geq A^{(\kappa-\delta)^{-(n-1)}}\right)+
    K \Big(A^{(\kappa-\delta)^{-n}}\Big)^{-\zeta}.
}
As a result we obtain that
    \eqn{
        \prob\left(S_n(0)\geq A^{(\kappa-\delta)^{-n}}\right)
    \leq \prob\left(S_1(0)\geq A^{1/(\kappa-\delta)}\right)+\sum_{k=2}^n K A^{-\zeta (\kappa-\delta)^{-k}},
    }
which can be made small by choosing $A\geq 1$ large enough.

Finally, by \refeq{bd-d-Sn} and when $n=(1-\eta)\frac{\log\log{|x|}}{|\log(\kappa)|}$,
then $A^{(\kappa-\delta)^{-n}}\leq |x|/2$, and we conclude that
$\prob(d(0,x)\leq 2n)=o(1)$.
\qed

\subsection{Lower bounds on distances for finite-variance degrees}
\label{sec-log-LB-dist}

We begin by establishing a general logarithmic lower bound valid for $\gamma>2$.

\begin{Theorem}[Logarithmic lower bound on distance for finite variance degrees]
\label{thm-log-dist-finvar}
Assume that there exists $\tau>1$
and $c>0$ such that \refeq{[1-F]-infvar-ub} holds and that $\gamma=\alpha(\tau-1)/d>2$.
Then, there exists an $\eta>0$ such that
    $$
    \lim_{|x|\rightarrow \infty}
    \prob(d(0,x)\geq \eta\log{|x|})=1.
    $$
\end{Theorem}

\proof We follow the proof of Theorem \ref{thm-crit-value-pos},
and obtain
    \begin{equation}\label{eq:dn_bd}
    \prob(d(0,x)=n)
    \leq \sum_{(x_1,\ldots, x_{n-1})}\prod_{i=1}^n
    g(|x_{i-1}-x_i|^{\alpha}/\lambda)^{1/2},
    \end{equation}
where we adopt the convention that $x_0=0$ and $x_n=x$. Define
    $$
    h(x)=(\log{|x|}+1)|x|^{-\alpha\big((\tau-1)/2\wedge 1\big)},
    $$
for $x\neq 0$ and $h(0)=0$. Then, using the bound in Lemma \ref{lem-g-bd} and the fact that the sum in (\ref{eq:dn_bd}) acts like a convolution, the right-hand side of (\ref{eq:dn_bd}) can be bounded by
    $$
    \prob(d(0,x)=n) \leq \big(C\lambda^{(\tau-1)/2}\big)^n h^{*n}(x),
    $$
where $h^{*n}$ denotes the $n$-fold convolution of $h$ with itself. Now, it is easy to see that
    \eqn{
    \lbeq{conv-bd}
    h^{*n}(x)
    \leq n\Big(\sup_{y\colon |y|\geq |x|/n}h(y)\Big) \Big(\sum_{u\neq 0}
    h(u)\Big)^{n-1}.
    }
Indeed, to see \refeq{conv-bd}, we note that
    \eqn{
    h^{*n}(x)=\sum_{x_1+\cdots+x_n=x} \prod_{i=1}^n h(x_i).
    }
When $x_1+\cdots+x_n=x$, there must be (at least one) $x_i$ with $|x_i|\geq |x|/n$.
We bound that factor by $\sup_{y\colon |y|\geq |x|/n}h(y)$, and sum out over
the remaining $x_j$ for $j\neq i$, noting that that sum is now unrestricted.

When $n\leq \eta\log{|x|}$, we can define $\kappa>0$ such that
    $$
    \sup_{y\colon |y|\geq |x|/n} h(y)
    \leq C' (\log{|x|})^{\kappa} |x|^{-\alpha\big((\tau-1)/2\wedge 1\big)}.
    $$
Furthermore, $\sum_{u\neq 0}h(u)<\infty$ when $\gamma=\alpha(\tau-1)/d>2$. As a result, we obtain that
    $$
    \prob(d(0,x)=n) \leq n\big(C\lambda^{\big((\tau-1)/2\wedge 1\big)}\big)^n (\log{|x|})^{\kappa} |x|^{-\alpha\big((\tau-1)/2\wedge 1\big)},
    $$
which is bounded by $|x|^{-\varepsilon}$ when $n\leq \eta\log{|x|}$ with $\eta>0$ sufficiently small.
This is true for any $n\leq \eta\log{|x|}$, so
    $$
    \prob(d(0,x)\leq \eta \log{|x|})
    \leq |x|^{-\varepsilon},
    $$
and the proof of Theorem \ref{thm-log-dist-finvar} is completed.
\qed
\medskip

\noindent
We next improve the above result to a polynomial lower bound when $\alpha>2d$.

\begin{Theorem}[Polynomial lower bound on distance for finite variance degrees when $\alpha>2d$]
\label{thm-poly-dist-finvar}
 Assume that there exists $\tau>1$
and $c>0$ such that \refeq{[1-F]-infvar-ub} holds, that $\gamma=\alpha(\tau-1)/d>2$
and that $\alpha>2d$. Then, for every $\vep<d[(\gamma \wedge \alpha/d)-2]/(d+1)$, we have
    $$
    \lim_{|x|\rightarrow \infty}
    \prob(d(0,x)\geq |x|^{\vep})=1.
    $$
\end{Theorem}

\proof We follow the proof of Theorem \ref{thm-log-dist-infvar-low}, that is, \refeq{bd-d-Sn}--\refeq{g1-bd},
and start by investigating
$\prob(S_n(0)\geq t)$. Now, for $t\to \infty$ and $n=o(t)$, we bound
    \eqan{
    \prob(S_n(0)\geq t)
    &\leq \prob(S_1(0)\geq t/n)+\sum_{k=1}^{n-1}
    \prob(S_k(0)\leq tk/n, S_{k+1}(0)\geq t(k+1)/n)\\
    &=o(1)+\sum_{k=1}^{n-1} \quad \sum_{u,v\colon |u|\leq tk/n, |v|\geq t(k+1)/n} g_1(|u-v|^{\alpha}/\lambda)\nn\\
    &\leq o(1) + K \sum_{k=1}^{n-1} \big(tk/n\big)^d \big(t/n)^{-\alpha [(\tau-1)\wedge 1]+d+\eta}
    \leq o(1) n^{d+1} t^{d[(2-\gamma \wedge \alpha/d)]+\eta},\nn
    }
where $\eta>0$ can be taken arbitrarily small.
This is $o(1)$ when $n\leq t^{\vep}$, where
\eqn{
\label{cond-pol}
\vep<\frac{d}{d+1}[(\gamma \wedge \alpha/d)-2].
}
According to \eqref{bd-d-Sn}, we have
$$\prob(d(0,x)\leq 2n)
    \leq \prob(S_n(0)\geq |x|/2),
$$
and hence $\prob(d(0,x)\leq |x|^{\vep})=o(1)$ for every $\vep$ satisfying \eqref{cond-pol}.
\qed

\section{Further work}
\label{sec-disc}

In this paper, we have studied degrees, percolation and distances in a long-range
percolation model with i.i.d.\ vertex weights. Using relatively simple tools, we have carved out the phase diagram by identifying appropriate bounds on degrees, critical values and distances as a function of the model parameters. Our model has power-law degrees, small- (and even ultra-small-) world behavior, with spatial connections on various spatial scales, and its properties depend in an intricate way on the number of finite moments of its degree distribution. The model shares many interesting features of \emph{both} inhomogeneous random graphs having power-law degrees, and long-range percolation. We remark that, while we have assumed that the edge probabilities has the precise form in \eqref{edgeprob},
it is not hard to see that our results extend to settings where $p_{xy}=h(\lambda W_xW_y/|x-y|^{\alpha})$, for some function $x\mapsto h(x)$ for which $x/2\leq h(x)\leq x,$ whenever $x\in [0,1]$. In the random graph setting, this is established in \cite{BRJ, Jans08a}.

There are a number of questions about the studied model that deserve further investigation.
In this section we mention some of them.\medskip

\noindent \textbf{Finiteness of the critical value.} In Theorem \ref{th:finite}, conditions are given that ensure percolation for large values of $\lambda$. In $d=1$ (with $\alpha\in(1,2]$), the condition is that the weights are bounded away from 0 and, in $d\geq 2$, that the weight distribution satisfies $F(0)<1$. These conditions are presumably not optimal and it would be interesting to investigate how far they can be relaxed. Indeed, some condition on the weight distribution near 0 is presumably necessary in order for percolation to be possible.
\medskip

\noindent \textbf{Distances.} We have given a logarithmic lower  bound on the graph distance when $\gamma>2$ and $\alpha>d$, a polynomial lower bound when $\gamma>2$ and $\alpha>2d$, and doubly logarithmic asymptotics when $\gamma<2$. These bounds, though, leave much room for improvement. When $\gamma\in (1,2)$ and $\tau>2$, it would be of interest to find the constant in front of the $\log\log|x|$. Is this constant $2/|\log(\gamma-1)|$ or $2/|\log(\alpha/d-1)|$ or something in between? Is the behavior for $\gamma>2$ and $\alpha\in (d,2d]$ really polylogarithmic, as it is in the deterministic case (see \cite{Biskup}, where Biskup proves that distances for long-range percolation are $\Theta((\log{|x|})^{\Delta+o(1)}),$
with $\Delta=1/\log_2(2d/\alpha)$)? When  $\gamma>2$ and $\alpha>2d$, can we identify the exponent $\mu$ such that $d(0,x)=|x|^{\mu+o(1)}$ whenever $0$ and $x$ are in the infinite component?\medskip

\noindent \textbf{Diameter in infinite mean case.} In \cite{BKPS}, it has been proved that, for long-range
percolation with infinite mean degrees (that is, when $W_x$ is constant and $\alpha<d$), then the diameter
of the infinite component is equal to $\lceil d/(d-\alpha)\rceil$. The proof crucially relies on
the notion of the \emph{stochastic dimension} for random relations in the lattice. It would be of
interest to investigate whether the diameter of the infinite component is bounded also in our model when $\gamma=\alpha(\tau-1)/d<1$.

\noindent \textbf{Critical behavior.}
The most interesting phenomena in percolation models can be found close to the critical
value. A central question is if the percolation function $\theta(\lambda)$ is continuous.
See \cite{Grim99} and \cite{BolRio06} for the rich history of this problem.
In \cite{Berg02}, it is shown that the percolation function is continuous when
$\alpha\in (d,2d)$. Is this also true in our model?
Further, percolation in two dimensions has received tremendous attention
in the past years, due to the connection to conformal invariance, see e.g.\ \cite{Smir01}.
In particular, the percolation function is continuous, and many critical
exponents are identified on the triangular lattice. The continuity of the percolation function extends to many finite-range percolation models in $d=2$. Is the percolation function
$\lambda\mapsto \theta(\lambda)$ also
continuous for our model? Further, how do the \emph{critical exponents}
depend on the randomness in the medium?
In \cite{HeyHofSak08}, the mean-field behavior of long-range percolation is
investigated, and it is shown that when $d>3((\alpha-d)\wedge 2)$, the model has mean-field
critical exponents. This raises the question what the upper-critical dimension in the presence of vertex weights is.

\noindent \textbf{Critical behavior on the torus.}
In random graph theory, there is recently a substantial interest in the critical
behavior of inhomogeneous random graphs of so-called rank-1, that is, the setting of our model
on the complete graph. See e.g.\ \cite{BhaHofLee09a, BhaHofLee09b,HatMol09,Hofs09a,Turo09}
for the relevant results. In this setting, we see that the critical behavior
when $\gamma>3$ is similar to that of the Erd\H{o}s-R\'enyi random graph
as identified in \cite{Aldo97}, while, for $\gamma\in (2,3)$,
it is rather different (see \cite{BhaHofLee09b,Hofs09a}).
This raises the question whether also for our inhomogeneous percolation model, the critical
behavior is different for $\gamma>3$ and for $\gamma\in (2,3)$. To best compare
the situations of (non-spatial) inhomogeneous random graphs and their spatial
counterparts, it would be useful to examine the setting on a finite torus.
Our model on the torus is translation invariant, and has a unique
critical value above which the largest connected component contains
a positive proportion of the vertices. It would be interesting
to investigate the \emph{critical behavior} of this spatial finite inhomogeneous
random graph.
\medskip

\noindent \textbf{Continuum analogues.} A continuum analogue of long-range percolation, known as the \emph{random connection model,} is described in \cite{MesRoy}. There, the vertex set is taken to be the points of a Poisson process on $\Rbold^d$ and two vertices $x$ and $y$ are connected by an edge with a probability given by a function $g$ of their separation $|x-y|$. An inhomogeneous version of this model is known as the Poisson Boolean model, or continuum percolation. Each Poisson point $x$ is then assigned a {\it random} radius $R_x$ and two points $x$ and $y$ are connected if $|x-y|\leq R_x+R_y$. Results on these models revolve around the existence of non-trivial critical intensity for the underlying Poisson process. There are no results so far on graph distances. It would be of interest to study a continuum version of the model in the setting of the current paper. This would constitute an alternative inhomogeneous formulation of the random connection model.\medskip

\paragraph{Acknowledgements.}
MD gratefully acknowledges (i) a visiting grant from NWO and STAR for a stay at
TU Delft from 1 October to 30 November 2010, and (ii) support from The Bank of Sweden Tercentenary Foundation. The work of RvdH was supported in part by NWO.


\begin{thebibliography}{99}
%\bibitem{Abramowitz}
%Abramowitz, M.~and Stegun,I.A.,
%\textit{Handbook of Mathematical Functions},
%Dover, New York, 1965.

\bibitem{AN}
Aizenman, M. and Newman, C.M., Discontinuity of the percolation density in one dimensional $1/|x-y|^2$ percolation models. \textit{Comm. Math. Phys.} \textbf{107}, 611-647, 1986.

\bibitem{Aldo97}
Aldous, D.,
\newblock Brownian excursions, critical random graphs and the multiplicative
  coalescent.
\newblock {\em Ann. Probab.} {\bf 25}, 812--854, 1997.


\bibitem{BKPS}
Benjamini, I., Kesten, H., Peres, Y. and Schramm, O., Geometry of the uniform spanning forest: transitions in diameters 4, 8, 12,\ldots. \textit{Ann. of Math.} \textbf{160}, 465-491, 2004.

\bibitem{Berg02}
Berger, N.,
\newblock Transience, recurrence and critical behavior for long-range
  percolation.
\newblock {\em Comm. Math. Phys.} {\bf 226}(3),531--558, 2002.

\bibitem{Berger}
Berger, N., A lower bound for chemical distances in sparse long-range percolation models. Preprint,  arXiv:math/0409021v1, 2004.

\bibitem{BhaHofLee09a}
Bhamidi, S., van~der Hofstad, R., and van Leeuwaarden, J.,
\newblock Scaling limits for critical inhomogeneous random graphs with finite
  third moments.
\newblock \textit{Elect. J. Probab.} {\bf 15}, 1682-1702, 2010.

\bibitem{BhaHofLee09b}
Bhamidi, S., van~der Hofstad, R. and van Leeuwaarden, J.,
\newblock Novel scaling limits for critical inhomogeneous random graphs.
\newblock Preprint 2009.


\bibitem{BGT}
Bingham, N.H., Goldie, C.M. and Teugels, J.L.,
\newblock \textit{Regular Variation.}
\newblock Encyclopedia of Mathematics and Its Applications, {\bf 27}, Cambridge University Press, Cambridge 1987.

\bibitem{Biskup}
Biskup, M., On the scaling of the chemical distance in long range percolation models, \textit{Ann. Probab.} \textbf{32}, 2933-2977, 2004.

\bibitem{BolRio06}
Bollob{\'a}s, B. and Riordan, O.,
\newblock {\em Percolation}.
\newblock Cambridge University Press, New York, 2006.

\bibitem{BRJ}
Bollob\'{a}s, B., Riordan, O. and Janson, S., The phase transition in inhomogeneous random graphs, \textit{Rand. Struct. Alg.} \textbf{31}, 3-122, 2007.


\bibitem{BDM-L}
Britton, T., Deijfen, M. and Martin-L\"{o}f, A., Generating simple random graphs with prescribed degree distribution, \textit{J. Stat. Phys.} \textbf{124}, 1377-1397, 2006.

\bibitem{CL:1}
Chung, F. and Lu, L., The average distances in random graphs with given expected degrees, \textit{Proc. Natl. Acad. Sci. USA} \textbf{99}, 15879-15882, 2002.

\bibitem{CL:2}
Chung, F. and Lu, L., Connected components in random graphs with given expected degree sequences, \textit{Ann. Comb.} \textbf{6}, 125-145, 2002.

\bibitem{CGS}
Coppersmith, D., Gamarnik, D. and Sviridenko, M., The diameter of a long-range percolation graph, \textit{Rand. Struct. Alg.} \textbf{21}, 1-13, 2002.

%\bibitem{EKM97}
%Embrechts, P., Kl\"uppelberg, C. and Mikosch, T.,
%\textit{Modelling Extreme Events}, Springer Verlag Berlin, 1997.

\bibitem{DomHofHoo10}
Dommers, S., van~der Hofstad, R., and Hooghiemstra, G.,
\newblock Diameters in preferential attachment graphs.
\newblock {\em J.\ Stat.\ Phys.}, {\bf 139}, 72--107, 2010.

\bibitem{Feller}
Feller, W., \textit{An Introduction to Probability Theory and its Applications,}
Vol.2, second edition, Wiley and Sons, New York, 1971.

\bibitem{GKN}
Gandolfi, A., Keane, M.S. and Newman, C.M., Uniqueness of the infinite component in a random graph with applications to percolation and spin glasses. \textit{Probab. Th. Rel. Fields} \textbf{92}, 511-527, 1992.

\bibitem{Grim99}
Grimmett, G.,
\newblock {\em Percolation}.
\newblock Springer, Berlin, 2nd edition, 1999.

\bibitem{HatMol09}
Hatami, H. and Molloy, M.,
\newblock {The scaling window for a random graph with a given degree sequence}.
\newblock Preprint 2009.


\bibitem{HeyHofSak08}
Heydenreich, M., van~der Hofstad, R. and Sakai, A.,
\newblock Mean-field behavior for long- and finite range ising model,
  percolation and self-avoiding walk.
\newblock {\em J. Statist. Phys.}, {\bf 132}(5), 1001--1049, 2008.

\bibitem{Hofs09a}
van~der Hofstad, R.,
\newblock Critical behavior in inhomogeneous random graphs.
\newblock Preprint 2009.

\bibitem{HHM}
van~der Hofstad, R., Hooghiemstra, G. and Van Mieghem, P., Distances in random graphs with finite variance degrees. \textit{Rand. Struct. Alg.} \textbf{26}, 76-123, 2005.

\bibitem{HHZ}
van der Hofstad, R., Hooghiemstra, G. and Znamenski, D., Distances in random graphs with finite mean and infinite variance degrees. \textit{Elect. J. Probab.} \textbf{12}, 703-766, 2007.

\bibitem{Jans08a}
Janson, S.,
\newblock Asymptotic equivalence and contiguity of some random graphs.
\newblock {\em Rand. Struct. Alg.}, {\bf 36}(1), 26-45, 2010.

\bibitem{LSS}
Liggett, T.M., Schonmann, R.H. and Stacey, A.M.,
Domination by product measure. \textit{Ann. Probab.} \textbf{25}, 71-95, 1997.

\bibitem{MesRoy}
Meester, R. and Roy. R, \textit{Continuum percolation}. Cambridge University Press, 1996.

\bibitem{NewSch}
Newman, C.M. and Schulman, L.S., One-dimensional $1/|j-i|^s$ percolation models: the existence of a transition for $s\leq 2$. \textit{Comm. Math. Phys.} \textbf{104}, 547-571, 1986.

\bibitem{NorrReit}
Norros, I. and Reittu, H., On a conditionally Poissonian graph process. \textit{Adv. Appl. Probab.} \textbf{38}, 59-75, 2006.

\bibitem{Sch}
Schulman, L.S., Long range percolation in one dimension, \textit{J. Phys. A} \textbf{16}, L639-L641, 1983.

\bibitem{Smir01}
Smirnov, S.,
\newblock Critical percolation in the plane: conformal invariance, {C}ardy's
  formula, scaling limits.
\newblock {\em C. R. Acad. Sci. Paris S\'er. I Math.} {\bf 333}(3), 239--244,
  2001.

\bibitem{Pieter}
Trapman, P., The growth of the infinite long-range percolation cluster. \textit{Ann. Probab.} \textbf{38},  1583-1608, 2010.

\bibitem{Turo09}
Turova, T.S.,
\newblock Diffusion approximation for the components in critical inhomogeneous
  random graphs of rank 1.
\newblock Preprint 2009.

\bibitem{Yukich07}
Yukich, J.E.,
Ultra-small scale-free geometric networks.
\textit{J. Appl. Prob.} \textbf{43}, 665-677, 2006.
\end{thebibliography}
\end{document}